\documentclass[12pt,a4paper]{amsart}
\usepackage{amssymb}
\usepackage{amsfonts}
\usepackage{amsthm}
\usepackage{amsmath}
\usepackage{amscd}
\usepackage[latin2]{inputenc}
\usepackage{t1enc}
\usepackage[mathscr]{eucal}
\usepackage{indentfirst}
\usepackage{graphicx}
\usepackage{graphics}
\numberwithin{equation}{section}
\usepackage[margin=2.9cm]{geometry}
\usepackage{epstopdf} 
\usepackage{xcolor}
 
\usepackage{verbatim}
 \usepackage{enumitem}
\usepackage{subfigure}
\theoremstyle{plain}
\newtheorem{theorem}{Theorem}[section]
\newtheorem{lemma}[theorem]{Lemma}

\newtheorem{proposition}[theorem]{Proposition}

 \theoremstyle{definition}
\newtheorem{definition}[theorem]{Definition}

\newtheorem{remark}[theorem]{Remark}
\newtheorem{?}[theorem]{Problem}
\newtheorem{example}[theorem]{Example}

\begin{document}


\title[Quantum chaos]{Disjoint data inverse problem on manifolds with quantum chaos bounds}


\author[] {Matti Lassas}
\address {Department of Mathematics and Statistics, University of Helsinki, Helsinki, Finland}
\email{matti.Lassas@helsinki.fi}

\author[]{Medet Nursultanov}
\address {Department of Mathematics and Statistics, University of Helsinki, Helsinki, Finland}
\email{medet.nursultanov@gmail.com}

\author[]{Lauri Oksanen}
\address {Department of Mathematics and Statistics, University of Helsinki, Helsinki, Finland}
\email{lauri.oksanen@helsinki.fi}

\author[]{Lauri Ylinen,}
\address {Department of Mathematics and Statistics, University of Helsinki, Helsinki, Finland}
\email{lauri.ylinen@helsinki.fi}



\begin{abstract}
	We consider the inverse problem to determine a smooth compact Riemannian manifold $(M,g)$ from a restriction of the source-to-solution operator, $\Lambda_{\mathcal{S,R}}$, for the wave equation on the manifold. Here, $\mathcal{S}$ and $\mathcal{R}$ are open sets on $M$, and $\Lambda_{\mathcal{S,R}}$ represents the measurements of waves produced by smooth sources supported on $\mathcal{S}$ and observed on $\mathcal{R}$. We emphasise that $\overline{\mathcal{S}}$ and $\overline{\mathcal{R}}$ could be disjoint. We demonstrate that $\Lambda_{\mathcal{S,R}}$ determines the manifold $(M,g)$ uniquely under the following spectral bound condition for the set $\mathcal{S}$: There exists a constant $C>0$ such that any normalized eigenfunction $\phi_k$ of the Laplace-Beltrami operator on $(M,g)$ satisfies
    \begin{equation*}
        1\leq C\|\phi_k\|_{L^2(\mathcal{S})}.
    \end{equation*}
    We note that, for the Anosov surface, this spectral bound condition is fulfilled for any non-empty open subset $\mathcal{S}$. Our approach is based on paper \cite{LassasOksanen2014} and the spectral bound condition above is an analogue of the Hassell-Tao condition there.
\end{abstract}

\maketitle



\tableofcontents

\section{Introduction}

Let $(M,g)$ be a smooth, connected and compact Riemannian manifold. Denote by $\Delta_g$ the Laplace-Beltrami operator on $M$.  We consider an inverse problem corresponding to the wave equation
\begin{equation}\label{wave eq}
	\begin{cases}
	(\partial_t^2 - \Delta_g) u (t,x) = f(t,x), & \text{in } (0,\infty)\times M; \\
	u\left.\right|_{t=0}= \partial_t u\left.\right|_{t=0} =0, & \text{in } M.
\end{cases}	
\end{equation}
We denote its solution by $u^f=u(t,x)$. For open and non-empty sets $\mathcal{R}$, $\mathcal{S}\subset M$, we define the restricted source-to-solution operator, 
\begin{equation*}\label{s_t_s_map}
	\Lambda_{\mathcal{S,R}} : f \mapsto u^f\left.\right|_{(0,\infty)\times \mathcal{R}},
	\qquad
	f\in C_0^\infty((0,\infty)\times \mathcal{S}).
\end{equation*}
The operator $\Lambda_{\mathcal{S,R}}$ models measurements for the wave equation with sources $f$ producing the wave on $(0,\infty)\times \mathcal{S}$ and the waves $u^f$ being observed on $(0,\infty)\times \mathcal{R}$. We investigate the inverse problem to determine $(M,g)$ from operator $\Lambda_{\mathcal{S,R}}$.

In the realm of the manifold with a boundary, analogous inverse problems are examined where $\mathcal{S}$ and $\mathcal{R}$ comprise portions of the boundary. These problems are commonly referred to as complete or partial boundary data problems and have been extensively investigated. Especially, the case when $\mathcal{S}=\mathcal{R}$ has been studied broadly, see \cite{Krein1951,Blagovestchenskii1969,Blagovestchenskii1971,Belishev1987,BelishevKurylev1992,LassasUhlmann2001} and references therein. Few results exist for the problem with disjoint partial data, we refer to works \cite{Rakesh2000,RakeshSacks2011,LassasOksanen2010,ImanuvilovUhlmannYamamoto2011}. 

On the other hand, the scenario where $\mathcal{S}$ and $\mathcal{R}$ represent non-empty open sets within the interior has received much less attention. To our knowledge, only one previous work has addressed this case: \cite{Helin_Lassas_Oksanen_Saksala2018} examines the situation where $\mathcal{S} = \mathcal{R}$. As a part of the results, the authors showed that the source-to-solution operator determines the manifold up to isomorphism. 

This kind of inverse problem with partial data is a common problem in many fields, such as physics, engineering, geology, and medical imaging. For instance, in geophysics, we may have only partial data about the earth's subsurface structure, and we may need to use techniques such as seismic tomography to estimate the distribution of materials and their properties.

\subsection{Main result.} 
In the present work, we focused on the case where $\mathcal{S}\cap\mathcal{R} =\emptyset$. It is worth noting that our method is applicable in situations where $\mathcal{S}$ and $\mathcal{R}$ are not disjoint. Our main result is the determination of the manifold $(M,g)$, up to isomorphism, from the source-to-solution operator $\Lambda_{\mathcal{S,R}}$, assuming that the following spectral bound condition holds: There is a constant $C_0>0$ such that for any basis of normalized eigenfunctions $\{\phi_j\}_{j\in \mathbb{N}}$ of the Laplace-Beltrami operator $\Delta_g$, the estimate
\begin{equation}\label{spec_cond}
	\tag{C}1\leq C_0 \|\phi_j\|_{L^2(\mathcal{S})}
\end{equation}
holds for $j\in \mathbb{N}$. It is observed that the aforementioned condition holds true for an Anosov surface, regardless of the specific choice of a non-empty open subset $\mathcal{S}$; see \cite{DyatlovJinNonnenmacher}. By Anosov surface, we refer to a compact connected Riemannian surface without a boundary, whose geodesic flow has the Anosov property. Anosov flows form a standard mathematical model of systems with strongly chaotic behaviour. Surfaces with negative Gauss curvature comprise a large class of examples.

Another sufficient condition for the spectral bound \eqref{spec_cond} is the exact controllability of $(M,g)$ from $\mathcal{S}$ in time $T>0$, that is the map
\begin{align*}
    U: L^2((0,T)&\times \mathcal S)\mapsto H^1(M)\times L^2(M)\\
    &U(f) = \left(u^f(T,\cdot),u^f(T,\cdot)\right)
\end{align*}
is surjective; see for instance \cite{HumbertPrivatTrelat}. It was also shown, in \cite{HumbertPrivatTrelat}, that exact controllability is satisfied under the geometric control condition: All geodesic rays, propagating in $M$, meet $\mathcal S$ within time $T$. We give a more detailed discussion in Section \ref{Sec_spect_bound}. 

As a main result, we prove the following:
\begin{theorem}\label{main_theorem}
    Let $(M,g)$ be a smooth, connected, compact Riemannian manifold. Let $\mathcal{R}$, $\mathcal{S}\subset M$ be open sets and $\Lambda_{\mathcal{S,R}}$ be the corresponding source-to-solution operator defined by \eqref{s_t_s_map}. Assume that $\mathcal{R}$ has a smooth boundary and the spectral bound condition \eqref{spec_cond} is satisfied for set $\mathcal{S}$. Then the data $(\mathcal{S},\left.g\right|_{\mathcal{S}})$, $(\mathcal{R},\left.g\right|_{\mathcal{R}})$, and $\Lambda_\mathcal{S,R}$ determine $(M,g)$ up to isometry. More precisely this means the following:

    Let $(M_1,g_1)$ and $(M_2,g_2)$ be smooth, connected and compact Riemannian manifolds, and let $\mathcal{R}_1$, $\mathcal{S}_1\subset M_1$ and $\mathcal{R}_2$, $\mathcal{S}_2\subset M_2$ be open non-empty sets. Assume that $\mathcal{R}_1$, $\mathcal{R}_2$ have smooth boundaries and the spectral bound condition \eqref{spec_cond} is satisfied for sets $\mathcal{S}_1$, $\mathcal{S}_2$ and there are isometries $\Phi: \overline{\mathcal{S}_1} \mapsto \overline{\mathcal{S}_2}$ and $\Psi: \overline{\mathcal{R}_1} \mapsto \overline{\mathcal{R}_2}$. Then, the identity
    \begin{equation*}
        \Lambda_{\mathcal{S}_1,\mathcal{R}_1}f = \Psi^* \Lambda_{\mathcal{S}_2,\mathcal{R}_2}\Phi_*f
    \end{equation*}
    implies that the manifolds $(M_1,g_1)$ and $(M_2,g_2)$ are isometric.
\end{theorem}
We note that Theorem \ref{main_theorem} holds if condition \eqref{spec_cond} is satisfied for $\mathcal{R}$ instead of $\mathcal{S}$. This follows from the fact that $\Lambda_{\mathcal{R,S}}$ is the adjoint of $\Lambda_{\mathcal{S,R}}$ composed with the time-reversal operator. Moreover, the smoothness of the boundary of $\mathcal R$ can be removed since, otherwise, we can consider a non-empty open subset of $\mathcal R$ with a smooth boundary. 

\subsection{Otline of the paper} The proof of the main results uses the idea of \cite{LassasOksanen2014}. The paper proceeds as follows: In Section \ref{Sec_spect_bound}, we discuss the spectral bound condition \eqref{spec_cond}, which is analogous to the Hassell-Tao condition for eigenvalues and eigenfunctions in \cite{LassasOksanen2014}. In Section \ref{Sec_weak_convergence}, we derive Blagovestchenskii's identity to compute the inner product of solutions $u^f(T,\cdot)$ and $u^h(T,\cdot)$ generated by smooth sources supported in $\mathcal{S}$ and $\mathcal{R}$, respectively. By using the spectral bound condition, we determine whether a sequence of waves with sources on $\mathcal{S}$ is $L^2$-bounded. Using Blagovestchenskii's identity and this analysis, we determine whether the given sequences of waves with sources on $\mathcal{R}$ converge weakly to zero. In Section \ref{Sec_rec_mon}, we establish a connection between the weak convergence of the sequences of waves and a geometric property -- a certain relation between domains of influence. This enables us to determine the distance between any point on the manifold $M$ and any point on $\partial \mathcal{R}$. Ultimately, it remains to determin the manifold $(M\setminus \mathcal{R},g)$ from the boundary distance function, which was solved in \cite{KatchalovKurylevLassas}.

\section{Spectral bound: Condition \eqref{spec_cond}}\label{Sec_spect_bound}
This section will discuss some sufficient conditions for \eqref{spec_cond}. 

It is known that in the case of Anosov surface, condition \eqref{spec_cond} is satisfied for any fixed open non-empty subset $\mathcal{S} \subset M$; see Theorem 1 in \cite{DyatlovJinNonnenmacher}. By Anosov surface, we refer to a compact connected Riemannian surface without boundary, whose geodesic flow
\begin{equation*}
	\Phi_t : T^1M \rightarrow T^1M
\end{equation*}
on the unit tangent bundle $T^1M$ of $M$, is Anosov type, that is, satisfies the condition: There is a $\Phi_t$-invariant splitting of the tangent bundle $TT^1M$ of $T^1M$ into subbundles of dimension $1$
\begin{equation*}
	TT^1M = E_s\oplus E_u \oplus E_1
\end{equation*}
such that $d\Phi_t\left.\right|_{E_s}$ is contracting, $d\Phi_t\left.\right|_{E_u}$ is expanding, and $d\Phi_t\left.\right|_{E_1}$ is generated by the flow vectors. This is the definition of geodesic of Anosov type on the Riemannian surface. For this definition in higher dimensional case, see for instance \cite{Klingenberg}.

It is well known that the geodesic flow on a manifold of strict negative sectional curvature is Anosov type; see \cite{Anosov1967,AnosovSinai1967,ArnoldAvez}.

\begin{example}
	Let $M = \Gamma\setminus \mathbb{H}^2$ be a convex cocompact quotient of hyperbolic space. Then, $M$ is a conformally compact manifold of constant negative curvature. Consequently, $M$ is an Anosov surface. 
\end{example}

\subsection{Exact controllability, observability, and geometric control condition}
Let us consider an operator, for $T>0$,
\begin{align}\label{C*_operator}
	&C^* : L^2((0,T)\times S) \mapsto H^1(M)\times L^2(M),\\
	\nonumber&C^*f(\cdot) = \left( - u^f(T,\cdot), \partial_t u^f(T,\cdot)\right),
\end{align}
where $u^f$ is the solution of 
\begin{equation*}
	\begin{cases}
		(\partial_t^2 - \Delta_g) u (t,x) = \tilde{f}(t,x) & \text{in } (0,\infty)\times M; \\
		u\left.\right|_{t=0}= \partial_t u\left.\right|_{t=0} =0 & \text{in } M.
	\end{cases}	
\end{equation*}
with $\tilde{f}$ being extension of $f$ to $M$ by zero. We also consider the dual system
\begin{align}\label{C_operator}
	&C : H^{-1}(M) \times L^2(M) \mapsto L^2((0,T)\times S),\\
	\nonumber& C(v_T, \omega_T) = \omega\left.\right|_{x\in S},
\end{align}
where $\omega$ is the solution of the equation
\begin{equation*}
	\begin{cases}
		(\partial_t^2 - \Delta_g)\omega(t,x) = 0,        & \text{in } (0,\infty)\times M; \\\
		 \omega(t,x) \left.\right|_{t=T}= \omega_T,      & \text{in } M\\
		 \partial_t\omega(t,x)\left.\right|_{t=T} = v_T,    & \text{in } M.
	\end{cases}
\end{equation*}
Note that $C^*$ is an adjoint operator of $C$. Indeed, by using integration by parts, one can check that
\begin{multline*}
	\int_{0}^{T}\int_{S} C(v_T,\omega_T)f(t,x)dxdt = \int_{0}^{T}\int_{M} \omega(t,x) \left(\partial_t^2 u^f (t,x)- \Delta_gu^f (t,x)\right) dxdt\\
	= \int_{M} 
	\begin{bmatrix}
		v_T(x)\\
		\omega_T(x)
	\end{bmatrix}
	\cdot
	\begin{bmatrix}
		- u^f(T,x)\\
		\partial_tu^f(T,x)
	\end{bmatrix}
	dx
	=\int_{M}(v_t(x),\omega_T(x)) \cdot C^*f(x)dx.
\end{multline*}

\begin{figure}
	\centering
	\subfigure[]{\includegraphics[width=0.4\textwidth]{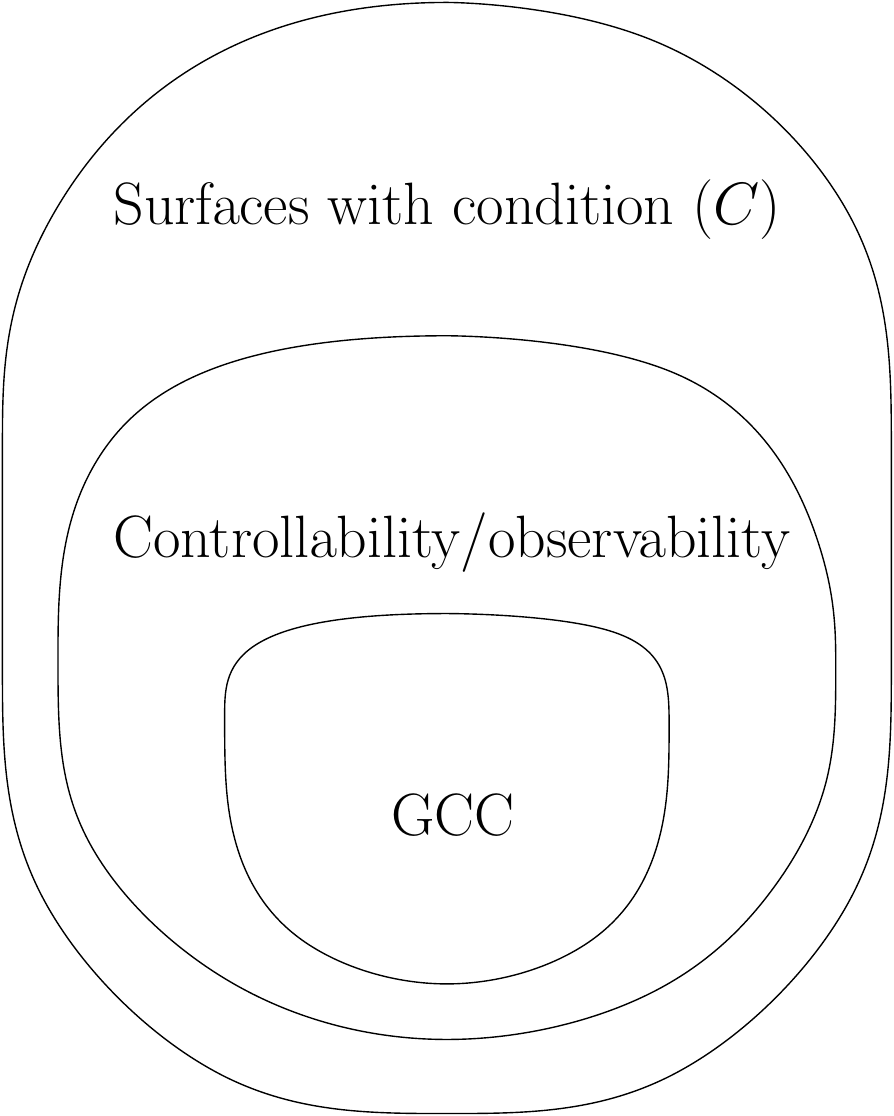}} 
	\hspace{20mm}
	\subfigure[]{\includegraphics[width=0.4\textwidth]{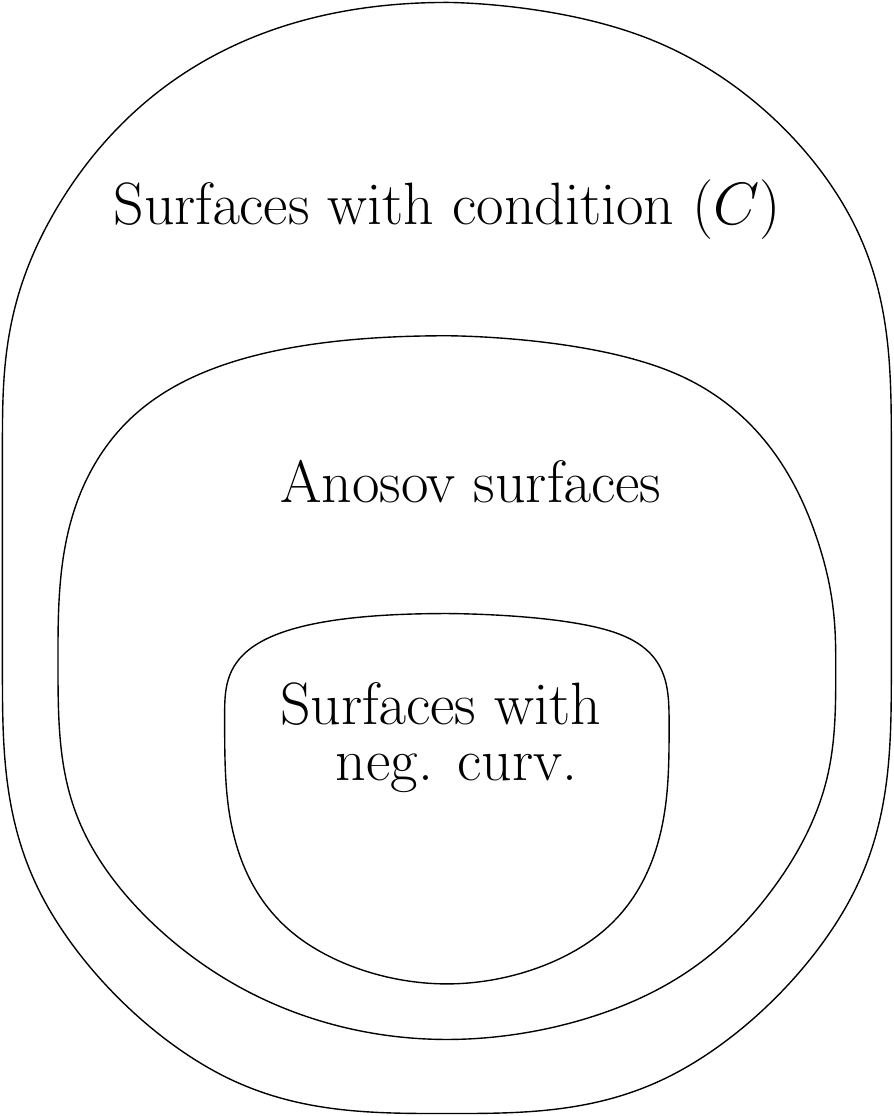}}
	\caption{}
	\label{fig:foobar}
\end{figure}

Now, we are ready to give definitions of the exact controllability and observability of the systems \eqref{C*_operator} and \eqref{C_operator}, respectively.

\begin{definition}
	The system \eqref{C*_operator} is exactly controllable if for any $(u_0,u_1) \in H^1(M)\times L^2(M)$, there exists $f\in L^2((0,T)\times S)$ such that the solution $u^f$ of \eqref{wave eq} satisfies $u^f(T,\cdot) = u_0$ and $\partial_t u^f(T,\cdot) = u_1$.
\end{definition}

\begin{definition}
	The system \eqref{C_operator} is observable if 
	\begin{equation*}
		\| \omega_T\|_{L^2(M)} + \|v_T\|_{H^{-1}(M)} \leq \|C(v_t,\omega_T)\|_{L^2((0,T)\times M)},
	\end{equation*}
or equivalently,
\begin{equation*}
	\| \omega(T,\cdot)\|_{L^2(M)} + \|\omega(T,\cdot)\|_{H^{-1}(M)} \leq \|\omega(\cdot,\cdot)\|_{L^2((0,T)\times M)}.
\end{equation*}
\end{definition}
The following theorem is a particular case of Theorem 2.1 in \cite{DoleckiRussell}:
\begin{theorem}
	The system \eqref{C*_operator} is exactly controllable if and only if the system \eqref{C_operator} is observable.
\end{theorem}

\begin{remark}
	The observability of \eqref{C*_operator} implies the spectral condition  \eqref{spec_cond}; see Theorem 1 in \cite{HumbertPrivatTrelat}. Consequently, the controllability of \eqref{C_operator} gives condition \eqref{spec_cond}.
\end{remark}

The following example, see \cite[page 753]{HumbertPrivatTrelat}, shows that the converse is not true.

\begin{example}[\textbf{A}]
	Let $M = \mathbb{T}^2$ and $S$ be the union of four triangles, each of the being at a corner of the square and whose side length is $1/2$. 
\end{example}

Another notion, related to observability property, is the Geometric Control Condition (GCC): All geodesic rays, propagating in $M$, meet $S$ within time $T$.  It is known that, for $S$ open, observability holds if the pair $(S,T)$ satisfy GCC; see for instance Corollary 3 in \cite{HumbertPrivatTrelat}. The converse is not true:

\begin{example}[\textbf{B}]
	Take $M=\mathbb{S}^2$, the unit sphere in $\mathbb{R}^3$, with standard metric and 
	$$S = \{(x,y,z)\in \mathbb{R}^3: x^2+y^2+z^2 = 0 \text{ and } z>0\}.$$
	In this case, for $T>\pi$ we have exact controllability for $(S,T)$, while GCC is violated; see Example B in \cite[page 15]{Lebeau}.
\end{example}
We summertime the discussion  of this section in Figure \ref{fig:foobar}.

\section{Testing weak convergence of sequences of waves}\label{Sec_weak_convergence}
In this section, we demonstrate that the source-to-solution map allows us to check if a sequence of waves converges weakly.
\subsection{Blagovestchenskii's identity}
We begin by obtaining the Blagovestchenskii identity, which was originally introduced in \cite{Blagovestchenskii1969,Blagovestchenskii1971}. Namely, we show that the value $\omega_{f,h}(T,T)$ can be recovered from $\Lambda_{\mathcal{R,S}}$, where
\begin{equation*}
	\omega_{f,h}(t,s):= \left(u^f(t,\cdot), u^h(s,\cdot )\right)_{L^2(M)},
\end{equation*}
and $h\in C_0^\infty( (0,T)\times \mathcal{R})$ and	$f\in C_0^\infty( (0,T)\times \mathcal{S})$.

To do this, we start with verifying that the adjoint of $\Lambda_{\mathcal{R,S}}$ is the operator $\Lambda_{\mathcal{S,R}}$ composed with the time-reversal operator defined as follows
\begin{equation*}
	R\psi(t): = \psi(T - t),
	\qquad 
	\psi\in C_0^\infty((0,T)).
\end{equation*}
The idea of the proof is based on Lemma 4.15 in \cite{KatchalovKurylevLassas}, we give it here for the convenience of the reader.
\begin{lemma}\label{LRLR}
	Let $h\in C_0^\infty( (0,T)\times \mathcal{S})$ and	$f\in C_0^\infty( (0,T)\times \mathcal{R})$, then
	\begin{equation*}
		\left(\Lambda_{\mathcal{R,S}}f,h\right)_{L^2((0,T)\times M)} = 	\left(f,R\Lambda_{\mathcal{S,R}}Rh\right)_{L^2((0,T)\times M)}.
	\end{equation*}
\end{lemma}
\begin{proof}
	Let $h\in C_0^\infty( (0,T)\times \mathcal{S})$ and	$f\in C_0^\infty( (0,T)\times \mathcal{R})$. Let $\omega$ be the solution of the equation
	\begin{equation*}
		\begin{cases}
			(\partial_t^2 - \Delta_g)\omega(t,x) = h(t,x);\\
			\partial_t\omega(t,x)\left.\right|_{t=T} = \omega(t,x) \left.\right|_{t=T}= 0.
		\end{cases}
	\end{equation*}
We note that $Ru^f$ solves is the solution of
\begin{equation*}
	\begin{cases}
		(\partial_t^2 - \Delta_g)Ru^f(t,x) = Rf(t,x), & \text{in } (0,\infty)\times M;\\
		\partial_tRu^f(T,x) = Ru^f(T,x) = 0;
	\end{cases}
\end{equation*}
and $R\omega$ is the solutions of
\begin{equation*}
	\begin{cases}
		(\partial_t^2 - \Delta_g)R\omega(t,x) = Rh(t,x), & \text{in } (0,\infty)\times M;\\
		\partial_tR\omega(0,x) = R\omega(0,x) = 0.
	\end{cases}
\end{equation*}
Therefore, by changing variables $t \mapsto T-t$, we obtain
\begin{align*}
	\int_{M}\int_{0}^{T} \Lambda_{\mathcal{R,S}}f(t,x)h(t,x)dtdV(x)& = \int_{M}\int_{0}^{T} R\Lambda_{\mathcal{R,S}}f(t,x)Rh(t,x)dtdV(x)\\
	& =\int_{M}\int_{0}^{T} R\Lambda_{\mathcal{R,S}}f(t,x)(\partial_t^2 - \Delta_g)R\omega(t,x)dtdV(x)\\
        & = \int_{M}\int_{0}^{T} Ru^f(t,x)(\partial_t^2 - \Delta_g)R\omega(t,x)dtdV(x).
\end{align*}
Then, by Green's identity, we obtain
\begin{align*}
	\left(\Lambda_{\mathcal{R,S}}f,h\right)_{L^2((0,T)\times M)} &= \int_{M}\int_{0}^{T} (\partial_t^2 - \Delta_g) Ru^f(t,x)R\omega(t,x)dtdV(x)\\
	& = \int_{M}\int_{0}^{T} Rf(t,x)\Lambda_{\mathcal{S,R}} Rh(t,x)dtdV(x)\\
	& = \int_{M}\int_{0}^{T} f(t,x) R\Lambda_{\mathcal{S,R}} Rh(t,x)dtdV(x).
\end{align*}
Here, we used that $R$ is symmetric.
\end{proof}

Now, we are ready to obtain the main result of this subsection.
\begin{lemma}\label{B identity}[Blagovestchenskii's identity]
	Let $h\in C_0^\infty( (0,T)\times \mathcal{R})$ and	$f\in C_0^\infty( (0,T)\times \mathcal{S})$, then
	\begin{equation*}
		\left(u^f(T,\cdot), u^h(T,\cdot)\right)_{L^2(M)} = \left(R\Lambda_{\mathcal{S,R}}RJf, h\right)_{L^2((0,T)\times \mathcal{R})} - \left( J\Lambda_{\mathcal{S,R}}f, h\right)_{L^2((0,T)\times \mathcal{R})},
	\end{equation*}
where 
\begin{equation*}
	J\psi(s) = \int_{s}^{2T - s} \psi(t) ds.
\end{equation*}
\end{lemma}
\begin{proof}
Let $h\in C_0^\infty( (0,T)\times \mathcal{R})$ and	$f\in C_0^\infty( (0,T)\times \mathcal{S})$. We write
\begin{multline*}
	(\partial_t^2 - \partial_s^2) \omega_{f,h}(t,s) = \\ 
	\left(f(t,\cdot) + \Delta_g u^f(t,\cdot) , u^h(s,\cdot)\right)_{L^2(M)} -  \left( u^f(t,\cdot) , h(s,\cdot) + \Delta_g u^h(s,\cdot)\right)_{L^2(M)}.
\end{multline*}
By using Green's identity, we write
\begin{align*}
	(\partial_t^2 - \partial_s^2) \omega_{f,h}(t,s) &= \left(f(t,\cdot) , u^h(s,\cdot)\right)_{L^2(M)} -  \left( u^f(t,\cdot) , h(s,\cdot) \right)_{L^2(M)}\\
	& = \left(f(t,\cdot) , \Lambda_{\mathcal{R,S}}h(s,\cdot)\right)_{L^2(M)} -  \left(\Lambda_{\mathcal{S,R}}f(t,\cdot) , h(s,\cdot) \right)_{L^2(M)}.
\end{align*}
Let us denote
\begin{equation*}
	F(t,s):=  \left(f(t,\cdot) , \Lambda_{\mathcal{R,S}}h(s,\cdot)\right)_{L^2(M)} -  \left(\Lambda_{\mathcal{S,R}}f(t,\cdot) , h(s,\cdot) \right)_{L^2(M)}.
\end{equation*}
Then, recalling the initial conditions, we see that $\omega_{f,h}$ is the solution of the one dimensional wave equation
\begin{equation*}
	\begin{cases}
		(\partial_t^2 - \partial_s^2) \omega_{f,h}(t,s) = F(t,s) & \text{on } [0,T]\times[0,T];\\
		\omega_{f,h}(t,s)\left.\right|_{t=0} = \partial_t\omega_{f,h}(t,s)\left.\right|_{t=0}= 0.
	\end{cases}
\end{equation*}
According to \cite{Evans2010}, its solution can be written as follows 
\begin{equation*}
	\omega_{f,h}(t,s) = \frac{1}{2} \int_{0}^{s} \int_{t - \tau}^{t + \tau} F(r,T - \tau) dr d\tau,
\end{equation*}
or, after changing of coordinates $\tau\mapsto T - \tau$,
\begin{equation*}
	\omega_{f,h}(t,s) = \frac{1}{2} \int_{T - s}^{T} \int_{t - T + \tau}^{t + T - \tau} F(r,\tau) dr d\tau.
\end{equation*}
In particular, we get
\begin{equation*}
	\omega_{f,h}(T,T) = \frac{1}{2} \int_{0}^{T} \int_{s}^{2T - s} F(t,s) dtds.
\end{equation*}
Recalling, the definition of $F$, we derive
\begin{multline*}
	2\omega_{f,h}(T,T) = \\
	\int_M\int_{0}^T \Lambda_{\mathcal{R,S}} h(s,x)  \int_{s}^{2T - s} f(t,x)dt   dsdV(x) - \int_M\int_{0}^T  h(s,x)  \int_{s}^{2T - s} \Lambda_{\mathcal{S,R}} f(t,x)dt   dsdV(x) .
\end{multline*}
Using Lemma \ref{LRLR}, we get
\begin{equation*}
	\omega_{f,h}(T,T)  = \left(R\Lambda_{\mathcal{S,R}}RJf, h\right)_{L^2((0,T)\times \mathcal{R})} - \left( J\Lambda_{\mathcal{S,R}}f, h\right)_{L^2((0,T)\times \mathcal{R})}.
\end{equation*}
\end{proof}

\subsection{Approximate controllability}
Here, we demonstrate that the approximate controllability of the wave equation can be derived from Tataru's unique continuation result \cite{Tataru}.

Let $\Gamma\subset M$ be an open non-empty subset. Then, the domain of influence of set $\Gamma$ is defined as follows
\begin{equation*}
	M(\Gamma, T) := \left\{x\in M: \; d_g(x,\Gamma)\leq T\right\}.
\end{equation*}
We also set
\begin{equation*}
	L^2(M(\Gamma, T)) := \left\{\phi\in L^2(M): \; \textrm{supp}\phi \subset M(\Gamma, T) \right\}.
\end{equation*}
It is easy to show that the space $L^2(M(\Gamma, T))$, equipped with $L^2$-norm, is a Hilbert space.
\begin{lemma}\label{applemma}
	Let $T>0$, then
	\begin{equation*}
		\mathcal{B}(\Gamma,T): = \left\{u^f(T,\cdot): \; f \in C_0^\infty((0,T)\times \Gamma )\right\}
	\end{equation*}
	is a dense subset of $L^2(M(\Gamma, T))$.
\end{lemma}
\begin{proof}
	Let $f\in C_0^\infty((0,T)\times \Gamma)$, then, by the finite speed of the wave propagation, we know that $u^f(T,\cdot)$ is supported in $M(\Gamma,T)$, so that $\mathcal{B}(\Gamma,T) \subset L^2(M(\Gamma, T))$. Since $L^2(M(\Gamma, T))$ is a Hilbert space, to prove density, it is sufficient to show that $\mathcal{B}(\Gamma,T)^{\perp} = \{0\}$, where
	\begin{equation*}
		\mathcal{B}(\Gamma,T)^{\perp} := \left\{ h \in  L^2(M(\Gamma, T)): \; (h,u)_{L^2(M)} = 0 \;\text{ for } u\in 	\mathcal{B}(\Gamma,T) \right\}.
	\end{equation*}
	
	Assume that $h\in \mathcal{B}(\Gamma,T)^{\perp}$, so that
	\begin{equation*}
		(u^f(T,\cdot), h(\cdot))_{L^2(M)} = 0
	\end{equation*}
	for all $f\in C_0^\infty((0,T)\times \Gamma)$. Let $\omega$ be the solution of the equation
	\begin{equation*}
		\begin{cases}
			(\partial_t^2 - \Delta_g)\omega = 0; & \text{in } (0,T)\times M;\\
			\omega(t,x)\left.\right|_{t=T} =  0; & \text{in } M;\\
			\partial_t\omega(t,x)\left.\right|_{t=T} =  h(x), & \text{in } M.
		\end{cases}
	\end{equation*}
	Then, we compute
	\begin{multline*}
		(f,\omega)_{L^2((0,T)\times M)} = ((\partial_t^2 - \Delta_g)u^f,\omega)_{L^2((0,T)\times M)}\\
		 = - \int_Mu^f(T,x)\partial_t\omega(T,x)dV(x) = (u^f(T,\cdot),h(\cdot))_{L^2(M)} = 0.
	\end{multline*}
	Since $f\in C_0^\infty((0,T)\times \Gamma)$ is an arbitrary function, it follows that $\omega = 0$ on $(0,T)\times \Gamma$. Let $\tilde{\omega}$ be the odd continuation of $\omega$ into $(T,2T)\times M$, that is $\tilde{\omega}(t,x) = -\omega(2T - t, x)$. Then $\tilde{\omega}$ solves
	\begin{equation*}
		\begin{cases}
			(\partial_t^2 - \Delta_g)\tilde{\omega} = 0; & \text{in } (T,2T)\times M;\\
			\tilde{\omega}(t,x)\left.\right|_{t=T}  = 0; & \text{in } M;\\
			\partial_t\tilde{\omega}(t,x) \left.\right|_{t=T} = h(x), & \text{in } M.
		\end{cases}
	\end{equation*}
	Moreover, since 
	\begin{equation*}
		\tilde{\omega}(T,x) = \omega(T,x) = 0
		\quad
		and 
		\quad
		\tilde{\omega}(t,x) \left.\right|_{t=T} = \partial_t\omega(t,x)\left.\right|_{t=T} =  h(x),
	\end{equation*}
	we can glue $\omega$ and $\tilde{\omega}$ to obtain the function
	\begin{equation*}
		W(t,x):= 
		\begin{cases}
			\omega(t,x) & \text{in } (0,T)\times M;\\
			\tilde{\omega}(t,x) & \text{in } (T,2T) \times M.
		\end{cases}
	\end{equation*}
	Then, $W$ solves the following equation
	\begin{equation*}
		\begin{cases}
			(\partial_t^2 - \Delta_g)W = 0, & \text{in } (0,2T)\times M;\\
			W(t,x)\left.\right|_{t=0} = \omega(t,x)\left.\right|_{t=0}, & \text{in } M;\\
			\partial_tW(t,x)\left.\right|_{t=0}= \partial_t\omega(t,x)\left.\right|_{t=0}, & \text{in } M,
		\end{cases}
	\end{equation*}
	and satisfies 
	\begin{equation*}
		W(t,x) = 0 \quad \text{for } (t,x) \in (0,2T)\times \Gamma.
	\end{equation*}
	Therefore, by Tataru's unique continuation result, see \cite{KatchalovKurylevLassas,Tataru}, we conclude that
	\begin{equation*}
		W = 0 \qquad 
		\text{in } \{(t,x)\in (0,2T)\times M: d_g(x,\Gamma)<\min(t, 2T-t)\}.
	\end{equation*}
In particular, at the level $t=T$, we have that
\begin{equation*}
    h(x) = \partial_tW(t,x)\left.\right|_{t=T} = 0 \qquad \text{for } x\in M(\Gamma,T).
\end{equation*}
 \end{proof}

\subsection{Testing weak convergence of sequences of waves}
The following lemma will be used on testing the weak convergence of the given waves with sources supported on $\mathcal{R}$.
\begin{lemma}\label{weak conv}
	Let $\Gamma\subset M$ be open, non-empty set, and let 
	$$T > \max_{x\in M} d_g(x,\Gamma).$$
	A sequence $\{v_l\}_{l\in \mathbb{N}}$ converges to zero weakly in $L^2(M)$ if and only if both (i) and (ii) hold, where
		\begin{itemize}
		\item[(i)] for all sequences $\{\psi_m\}_{m\in \mathbb{N}}\subset C_0^\infty((0,T)\times \Gamma)$ such that $\{u^{\psi_m}(T,\cdot)\}_{m\in \mathbb{N}}$ is bounded in $L^2(M)$, there is $C>0$ satisfying 
		\begin{equation*}
			|(v_l(\cdot), u^{\psi_m}(T,\cdot))_{L^2(M)}| \leq C \qquad \text{for all } l,m\in \mathbb{N}.
		\end{equation*}
	\item[(ii)] $\lim_{l\rightarrow\infty}(v_l(\cdot), u^{\psi}(T,\cdot))_{L^2(M)} = 0$ for all $\psi \in C_0^\infty((0,T)\times \Gamma)$.
	\end{itemize}
\end{lemma}

\begin{proof}
	Assume that $\{v_l\}_{l\in \mathbb{N}}$ converges to zero weakly in $L^2(M)$. Then, $(ii)$ follows from the definition and $(i)$ follows from the fact that a weakly convergent sequence is bounded. 
	
	Conversely, assume that conditions $(i)$ and $(ii)$ hold. Let $\omega\in L^2(M)$. Since $T > \max_{x\in M} d_g(x,\Gamma)$, we have that $\omega\in M(\Gamma,T)$. Hence, by Lemma \ref{applemma}, there is a sequence of sources $\{\psi_m\}_{m\in \mathbb{N}} \subset C_0^\infty((0,T)\times \Gamma )$ such that $\{u^{\psi_m}(T,\cdot)\}_{m\in \mathbb{N}}$ converges to $\omega$ in $L^2(M)$. Therefore, by assumption $(i)$,
	\begin{equation*}
		|(v_l,\omega)_{L^2(M)}|= \lim_{m\rightarrow\infty} |(v_l(\cdot),u^{\psi_m}(T,\cdot))_{L^2(M)}|.
	\end{equation*}
	This implies that $\{v_l\}_{l\in \mathbb{N}}$ is bounded in $L^2(M)$. Next, we write
	\begin{equation*}
		|(v_l,\omega)_{L^2(M)}| \leq \sup_{l\in \mathbb{N}} \|v_l\|_{L^2(M)} \| \omega(\cdot) - u^{\psi_m}(T,\cdot)\|_{L^2(M)} + |(v_l,u^{\psi_m}(T,\cdot))_{L^2(M)}|.
	\end{equation*}
	Consequently, taking $m$ and $l$ large enough, we can make the right-hand side as small as we need. This implies that $\{v_l\}_{l\in \mathbb{N}}$ weakly converges to zero.
\end{proof}

Let $\{\lambda_j\}_{j\in \mathbb{N}}$ be the eigenvalues of Laplace operator on $M$ indexed with the increasing order. We denote by $\{K_j\}_{j\in \mathbb{N}}$ the corresponding multiplicities. Let 
$$\{\{\phi_{jk}\}_{k=1}^{K_j}\}_{j\in\mathbb{N}}$$
the orthonormal basis of eigenvalues numerated such that $\{\phi_{jk}\}_{k=1}^{K_j}$ correspond to eigenvalue $\lambda_j$. Finally, we set
\begin{equation*}
	E_j := \textrm{span}\{\phi_{jk}\left.\right|_{\mathcal{S}}: \; k=1,\cdots,K_j\}.
\end{equation*}

Next, we prove an auxiliary lemma, which will be used later.
\begin{lemma}\label{scalprod}
	Let $\psi \in C_0^\infty((0,\infty)\times\mathcal{S})$, then we have
	\begin{equation*}
		\left(u^\psi(T,\cdot), \phi_{jk}(\cdot)\right)_{L^2(M)} = \int_{0}^{T} \int_M s_j(t)\psi(t,x) \phi_{jk}(x)dV(x)dt
	\end{equation*}
where $s_j(t):= \sin\left(\sqrt{\lambda_j}(T - t)\right)/\sqrt{\lambda_j}$.
\end{lemma}
\begin{proof}
	We check
	\begin{align*}
		\partial_t^2 	\left(u^\psi(t,\cdot), \phi_{jk}(\cdot)\right)_{L^2(M)} &= \int_{M} \left(\Delta_g u^\psi(t,x) + f(t,x)\right)\phi_{jk}(x)dV(x)\\
		& = \lambda_j \int_M u^\psi(t,x)\phi_{jk}(x)dV(x) + \int_{M} \psi(t,x)\phi_{jk}dV(x).
	\end{align*}
	Moreover, using initial conditions, we know that
	\begin{equation*}
		\left(u^\psi(0,\cdot), \phi_{jk}(\cdot)\right)_{L^2(M)} = 0, \qquad \partial_t\left(u^\psi(0,\cdot), \phi_{jk}(\cdot)\right)_{L^2(M)} = 0. 
	\end{equation*}
	Solving the one-dimensional Helmholtz equation, we obtain
	\begin{equation*}
		\left(u^\psi(t,\cdot), \phi_{jk}(\cdot)\right)_{L^2(M)}  = \int_{0}^{t} \int_M s_j(t)\psi(t,x) \phi_{jk}(x)dV(x)dt.
	\end{equation*}
	By choosing $t=T$, we obtain the identity.
\end{proof}
The meromorphic mapping of the Fourier transform of the operator $\Lambda_{\mathcal{S,R}}$ with respect to the time variable enables the determination of eigenvalues $\{\lambda_j\}_{j\in \mathbb{N}}$ and spaces $\{E_j\}_{j\in \mathbb{N}}$ through the analysis of its poles and residues. This and Lemmas \ref{B identity} and \ref{scalprod}, allows us to check $L^2$-boundedness of the sequence of waves:
\begin{proposition}\label{L2boundedness}
	Let $T>0$ and $\{\psi_m\}_{m\in \mathbb{N}} \subset C_0^\infty((0,\infty)\times\mathcal{S})$. Assume that there is $C_0> 0$ such that
	\begin{equation*}
		1 < C_0 \|\phi_{jk}\|_{L^2(\mathcal{S})}, \quad \text{for } j\in \mathbb{N} \text{ and } k=1,\cdots, K_j.
	\end{equation*}
	Then the following conditions are equivalent:
	\begin{itemize}
		\item[(i)] The sequence $\{u^{\psi_m}(T,\cdot)\}_{m\in \mathbb{N}}$ is bounded in $L^2(M)$.\\
		\item[(ii)] For all $C_1>0$ and $\{e_j\}_{j\in N}\subset C^\infty(\mathcal{S})$ satisfying
		\begin{equation}\label{ej}
			e_j\in E_j, \qquad \|e_j\|_{L^2(\mathcal{S})} \leq C_1,
		\end{equation} 
	there is $C_2 > 0$ such that 
	\begin{equation*}
		\sum_{j=1}^{\infty} \left(\int_{0}^{T}\int_{M}s_j(t)\psi_m(t,x)e_j(x)dV(x)dt\right)^2 \leq C_2,
	\end{equation*}
for all $m\in \mathbb{N}$.
	\end{itemize}
\end{proposition}

\begin{proof}
	Assume that (i) holds and let $\{e_j\}_{j\in\mathbb{N}}$ be a sequence satisfying \eqref{ej}. Without loss of generality, we may assume that 
	\begin{equation*}
		e_j = \left.c_j \phi_{j1}\right|_{\mathcal{S}},
	\end{equation*}
for some constant $c_j > 0$. By hypothesis, we estimate
\begin{equation*}
		\frac{c_j}{C_0} \leq c_j\|\phi_{j1}\|_{L^2(M)} = \|e_j\|_{L^2(\mathcal{S})}\leq C_1.
\end{equation*}
so that,
\begin{equation*}
	c_j\leq C_0C_1 \quad \text{for } j\in \mathbb{N}. 
\end{equation*}
Using Lemma \ref{scalprod}, we write
\begin{multline}
	\sum_{j=1}^{\infty} \left(\int_{0}^{T}\int_{M}s_j(t)\psi_m(t,x)e_j(x)dV(x)dt\right)^2 = \sum_{j=1}^{\infty} c_j^2 \left(u^{\psi_m}(T,\cdot), \phi_{j1}(\cdot)\right)_{L^2(M)}^2\\
	\leq (C_0C_1)^2\left\|u^{\psi_m}(T,\cdot))\right\|_{L^2(M)}^2.
\end{multline}

Conversely, assume that (ii) holds. Let $v\in L^2(M)$. Let $P_j$ be the orthogonal projection onto the $j$th eigenspace. By rotating the basis, we may assume that 
\begin{equation*}
	\phi_{j1} = \frac{P_jv}{\|P_jv\|_{L^2(M)}}.
\end{equation*}
Let us choose $e_j = \phi_{j1}$, then we estimate
\begin{equation}
	\|e_j\|_{L^2(\mathcal{S})} =  \|\phi_{j1}\|_{L^2(\mathcal{S})} \leq  \|\phi_{j1}\|_{L^2(M)} = 1,
\end{equation}
so that $\{e_j\}_{j\in \mathbb{N}}$ satisfies the conditions of (ii).  Therefore, by using Lemma \ref{scalprod},
\begin{align*}
	|(u^{\psi_m}(T,\cdot), v)|^2 &= \left|\sum_{j=1}^{\infty}(v,\phi_{j1})_{L^2(M)}\left(u^{\psi_m}(T,\cdot),\phi_{j1}(\cdot)\right)_{L^2(M)}\right|^2\\
	& \leq \sum_{j=1}^{\infty}(v,\phi_{j1})_{L^2(M)}^2 \sum_{j=1}^{\infty} \left(\int_{0}^{T}\int_{M}s_j(t)\psi_m(t,x)e_j(x)dV(x)dt\right)^2\\
	& \leq C_2 \|v\|_{L^2(M)}^2.
\end{align*}
Since $v$ was arbitrary, this implies that $\{u^{\psi_m}(T,\cdot)\}_{m\in \mathbb{N}}$ is bounded in $L^2(M)$.
\end{proof}

\section{Reconstruction of the manifold}\label{Sec_rec_mon}
In this section, we show that $\Lambda_{\mathcal{S}, \mathcal{R}}$ determines $(M,g)$. This is accomplished by establishing the distance function, which will be specified at a later point. With the determination of the distance function, the complete manifolds can then be reconstructed using the procedures outlined in Section 3.8 in \cite{KatchalovKurylevLassas}. 

\subsection{From the weakly convergent sequence of waves to the relation between the domain of influences}
\begin{lemma}\label{ver_intersection}
	Let $T>0$, $K\in \mathbb{N}$, and $\{\Gamma_k\}_{k=0}^{K}$ be open non-empty sets in $M$. Let $\{s_k\}_{k=0}^{K}$ be positive numbers not exceeding $T$. Then, the following properties are equivalent
	\begin{itemize}
		\item[(i)] $M(\Gamma_0,s_0)\subset \bigcup_{k=1}^K M(\Gamma_k,s_k)$.
		\item[(ii)] For any $f_0\in C_0^{\infty}((T-s_0,T)\times \Gamma_0)$ there exists a sequence
    \begin{equation*}
        \{f_j\}_{j\in\mathbb{N}}\subset  C_0^{\infty}\left(\bigcup_{k=1}^{K} (T-s_k, T)\times \Gamma_k\right)
    \end{equation*}
    such that 
		\begin{equation*}
			u^{f_j}(T,\cdot) \rightarrow u^{f_0}(T,\cdot)
		\end{equation*}
	weakly in $L^2(M)$ as $j\rightarrow \infty$.
	\item[(ii')] For any $f_0\in C_0^{\infty}((T-s_0,T)\times \Gamma_0)$ there exists a sequence
    \begin{equation*}
        \{f_j\}_{j\in\mathbb{N}}\subset  C_0^{\infty}\left(\bigcup_{k=1}^{K} (T-s_k, T)\times \Gamma_k\right)
    \end{equation*}
    such that 
		\begin{equation*}
			u^{f_j}(T,\cdot) \rightarrow u^{f_0}(T,\cdot)
		\end{equation*}
	strongly in $L^2(M)$ as $j\rightarrow \infty$.
	\end{itemize}
\end{lemma}
\begin{proof}
	Assume that $(i)$ holds. By the finite speed of the wave propagation, $u^{f_0}(T, \cdot)$ is supported in $M(\Gamma_0,s_0)$, and hence, Lemma \ref{applemma} implies $(ii)$ and $(ii')$.
	
	It remains to show that $(ii)$ implies $(i)$. Assume that $(i)$ does not hold, so that 
    \begin{equation*}
        M(\Gamma_0,s_0) \not\subset \bigcup_{k=1}^K M(\Gamma_k,s_k).
    \end{equation*}
    Moreover, since $\bigcup_{k=1}^K M(\Gamma_k,s_k)$ is a closed set, it follows that
    \begin{equation}\label{difference}
        M(\Gamma_0,s_0)^{int} \not\subset \bigcup_{k=1}^K M(\Gamma_k,s_k),
    \end{equation}
    where
	\begin{equation*}
		M(\Gamma_0,s_0)^{int} := \left\{ x\in M: d_g(x,\Gamma_0) < s \right\}.
	\end{equation*}
	Therefore, the difference \eqref{difference} is non-empty. Moreover, since $M(\Gamma_0,s_0)^{int} $ is open and $\bigcup_{k=1}^K M(\Gamma_k,s_k)$ is closed, we know that the difference \eqref{difference} is an open set. Hence, we conclude that there exists an open non-empty set 
    \begin{equation*}
        U\subset M(\Gamma_0,s_0)\setminus M(\Gamma,s).
    \end{equation*}
    By Lemma \ref{applemma}, there exists $f_0\in C_0^{\infty}((T-s_0,T)\times \Gamma_0)$ such that
	\begin{equation}\label{ver_intersection_0}
		\left( u^{f_0}(T,\cdot), 1_{U}(\cdot) \right)_{L^2(M)} \neq 0.
	\end{equation}
    where $1_{U}(\cdot)$ the indicator function of the set $U$. Let $f \in C_0^{\infty}\left(\bigcup_{k=1}^{K} (T-s_k, T)\times \Gamma_k\right)$, then, by finite speed of the wave propagation $u^f(T, \cdot)$ is supported in $\bigcup_{k=1}^K M(\Gamma_k,s_k)$. Therefore, from \eqref{ver_intersection_0}, it follows that
	\begin{equation*}
		\left(u^{f_0}(T_0,\cdot) - u^{f}(T,\cdot)_{L^2(M)}, 1_{U}(\cdot)\right)_{L^2(M)}\neq 0,
	\end{equation*}
    and hence, $(ii)$ does not hold.
\end{proof}

Next, we prove the following auxiliary lemma:
\begin{lemma}\label{aux_openset_poin}
    Let $\Gamma_2$ be an open, non-empty set in $M$ with a smooth boundary. Let $y_0$, $y_1 \in \partial \Gamma_2$ and $s_0$, $s_1$, $s_2>0$. Then the following properties are equivalent:
    \begin{itemize}
        \item[(i)] $M(y_0, s_0)\subset M(y_1, s_1)\cup M(\Gamma_2, s_2)$.
        \item[(ii)] For any $\varepsilon>0$ there exists an open set $\Gamma_0\subset \Gamma_2$ such that $y_0\in \overline{\Gamma_0}$ and
        \begin{equation*}
            M(\Gamma_0, s_0)\subset M(\Gamma_1, s_1 + \varepsilon)\cup M(\Gamma_2, s_2 + \varepsilon)
        \end{equation*}
        for all $\Gamma_1\subset \Gamma_2$ open sets such that $y_1\in \overline{\Gamma_1}$.
    \end{itemize}
\end{lemma}
\begin{proof}
    Assume that $(i)$ is valid. Let $\varepsilon>0$ and 
    $$\Gamma_0 = \{x\in \Gamma_2: d_g(x, y_0) < \varepsilon/2\}.$$ 
    Then
    \begin{equation*}
        M(\Gamma_0, s_0)\subset M(\Gamma_1, s_1 + \varepsilon)\cup M(\Gamma_2, s_2 + \varepsilon)
    \end{equation*}
    holds for any $\Gamma_1\subset \Gamma_2$ open neighbourhood of $y_1$.

    Next, assume that $(i)$ is not valid. Then, there is $z\in M(y_0,s_0)$ such that
    \begin{equation*}
        \varepsilon = \frac{1}{3} d_g\big(z, M(y_1, s_1)\cup M(\Gamma_2, s_2)\big) >0.
    \end{equation*}
    We set 
    $$\Gamma_1 = \{x\in \Gamma_2: \; d_g(x,y_1)<\varepsilon\},$$ 
    then
    \begin{equation*}
        z\notin M(\Gamma_1, s_1 + \varepsilon)\cup M(\Gamma_2, s_2 + \varepsilon).
    \end{equation*}
    However, $z\in M(\Gamma_0, s_0)$ for any $\Gamma_0\subset \Gamma_2$ open neighbourhood of $y_0$, so that $(ii)$ is not fulfilled.
\end{proof}

We end this section by summarizing our method to extract geometric information from the data in the following proposition:
\begin{proposition}\label{recov_relation}
    Let $\mathcal{S}$, $\mathcal{R}\subset M$ be open, non-empty sets and suppose that condition \eqref{spec_cond} is satisfied with the set $\mathcal{S}$. Assume that $\mathcal{R}$ has a smooth boundary. Then the source-to-solution map $\Lambda_{\mathcal{S},\mathcal{R}}$ and the smooth structure of $\overline{\mathcal{S}}\cup \overline{\mathcal{R}}$ determine the relation
    \begin{equation}\label{relation_for_distance}
        \left\{(y_0,y_1, s_0,s_1,s_2)\in \partial \mathcal{R}^2\times (0,\infty)^3: M(y_0,s_0) \subset M(y_1,s_1)\cup M(\mathcal{R},s_2)\right\}.
    \end{equation}
\end{proposition}
\begin{proof}
    For a sequences $\{\psi_k\}_{k\in \mathbb{N}} \subset C_0^\infty((0,\infty)\times\mathcal{S})$, by Proposition \ref{L2boundedness}, we can determine if the corresponding sequence $\{u^{\psi_k}(T,\cdot)\}_{k\in \mathbb{N}}$ is bounded in $L^2(M)$. Hence, by Lemmas  \ref{B identity} and \ref{weak conv}, we can determine for which sequences $\{f_k\}_{k\in \mathbb{N}} \subset C_0^\infty((0,\infty)\times\mathcal{R})$ the corresponding sequence $\{u^{f_k}(T,\cdot)\}_{k\in \mathbb{N}}$ converges weakly to zero. Therefore, due to Lemma \ref{ver_intersection}, for any choice of $\varepsilon$, $s_0$, $s_1$, $s_2>0$ and open sets $\Gamma_0$, $\Gamma_1\subset \mathcal{R}$ such that $y_0\in \overline{\Gamma_0}$ and $y_1\in \overline{\Gamma_1}$, we can verify if condition $(i)$ of Lemma \ref{ver_intersection} is fulfilled with $\Gamma_2 = \mathcal{R}$. Therefore, by Lemma \ref{aux_openset_poin}, we can determine the relation \eqref{relation_for_distance}.
\end{proof}

 \subsection{Recovery of the distance functions}
Here, we determine the distances from any point of the manifold to the boundary points $\partial\mathcal{R}$. To accomplish this, we introduce an auxiliary function: Let $\Gamma\subset M$ be an open set with smooth boundary and $\nu$ be the unit normal vector on $\partial \Gamma$ pointing into the interior of  $M\setminus\Gamma$. For $y \in \partial \Gamma$, let us define
\begin{equation*}
	\sigma_\Gamma(y) := \sup\left\{ s>0: \; d_g(\gamma(s;y,\nu), \Gamma) = s \right\},
\end{equation*}
where $\gamma(\cdot;y,\nu)$ is the geodesic with the initial data $\gamma(0;y,\nu) = y$ and $\dot{\gamma}(0;y,\nu) = \nu_y$.

We aim to recover the distances
\begin{equation}\label{distances}
	d_g(\gamma(s;y,\nu), z), 
	\qquad
	(s,y)\in \mathcal{N}_{\mathcal{R}} \text{ and } z\in \partial \mathcal{R},
\end{equation}
where
\begin{equation*}
	\mathcal{N}_{\mathcal{R}} := \left\{ (y, s) \in \partial \mathcal{R}\times (0,\infty): \; s\leq \sigma_\mathcal{R}(y)\right\}.
\end{equation*}
 In the following lemma, we show that $\sigma_{\mathcal{R}}$ can be recovered from the data.
\begin{lemma}
	Let $\Gamma \subset M$ be an open domain with a smooth boundary. Let $r>0$ and $y\in \partial \Gamma$. Then the following properties are equivalent
	\begin{itemize}
		\item[(i)] $r\leq $ $\sigma_\Gamma(y)$;
		\item[(ii)] for any $0<s<t\leq r$, $M(y,t) \not\subset M(\Gamma, s)$.
	\end{itemize}
\end{lemma}

\begin{proof}
	Assume that $(i)$ holds. Let $\varepsilon>0$ be small enough so that $s< t - \varepsilon$. Since $t - \varepsilon < \sigma_\Gamma(y)$, we know that 
	\begin{equation*}
		d_g(\Gamma, \gamma(t - \varepsilon;y,\nu)) = t - \varepsilon,
	\end{equation*}
and hence, $\gamma(t - \varepsilon; y,\nu) \not\in M(\Gamma, s)$, but $\gamma(t - \varepsilon; y,\nu) \in M(y,t)$.

Next, assume that $(ii)$ holds. In particular,
\begin{equation*}
    M(y,r)\not\subset M(\Gamma, s), \qquad \text{for any } s<r.
\end{equation*}
Therefore, there exists $x\in M$ such that $d_g(x,y) = r$ and $x\notin M(\Gamma,s)$ for any $s<r$, or equivalently,
\begin{equation*}
    d_g(x,\Gamma) = d_g(x,y) = r.
\end{equation*}
Let $\gamma$ be a unit-speed shortest path from $y$ to $x$. Since $\partial \Gamma$ is smooth, it follows that $\dot{\gamma}(0) = \nu_y$, and hence, $\gamma(s)$ coincides with $\gamma(s;y,\nu)$. Therefore, we conclude that $\gamma(r; y,\nu) = x$ and $d_g(\gamma(r;y,\nu),y) = r$, so that $r\leq \sigma_\Gamma(y)$.
\end{proof}
Now, we will demonstrate that the distances \eqref{distances} can be determined through the function $\sigma_{\mathcal{R}}$ and the relation \eqref{relation_for_distance}.
\begin{lemma}\label{rec_dist_funct}
	Let $\Gamma \subset M$ be an open domain with a smooth boundary. Let $y\in \partial \Gamma$ and $y_1\in M$. Assume that $t>0$ and $0<s<\sigma_\Gamma(y)$, then the following properties are equivalent:
	\begin{itemize}
		\item[(i)]  $d(\gamma(s;y,\nu), y_1) \leq t$;
		\item[(ii)] For any $\varepsilon>0$ there exists $\delta > 0$ such that 
		\begin{equation*}
			M(y,s) \subset M(\Gamma, s - \delta) \cup M(y_1, t + \varepsilon).
		\end{equation*}
	\end{itemize}
\end{lemma}

\begin{proof}
    Assume that $(i)$ holds, $d_g(y_1, \gamma(s;y,\nu)) \leq t$. Let $\varepsilon > 0$. Then, 
    \begin{equation*}
        M(\gamma(s;y,\nu),\varepsilon) \subset M(y_1, t + \varepsilon).
    \end{equation*}
    Let us define
    \begin{equation*}
        X_n := \overline{ M(y,s) \setminus M(\Gamma,s - 1/n)}.
    \end{equation*}
    Since $\partial\Gamma$ is smooth,
    \begin{equation}\label{intersection_of_X}
        \bigcap_{n\in \mathbb{N}} X_n = \gamma(s;y,\nu).
    \end{equation}
    Let $x_n \in X_n$ be such that
    \begin{equation*}
        d_g(x_n, \gamma(s;y,\nu)) \geq d_g(x, \gamma(s;y,\nu)) \quad \text{for all } x\in X_n.
    \end{equation*}
    Note that such $x_n$ exists since $X_n$ is compact and the distance function is continuous. Let $\Tilde{x}$ be a limit point of $\{x_n\}_{n\in\mathbb{N}}$. Since $X_n$ is closed, we conclude that $\Tilde{x} \in X_n$ for all $n\in \mathbb{N}$. Then, by \eqref{intersection_of_X}, we know that $\Tilde{x} = \gamma(s;y,\nu)$. Therefore, for sufficiently large $n_\varepsilon \in \mathbb{N}$, it follows that $d_g(x_{n_\varepsilon}, \gamma(s;y,\nu)) < \varepsilon$ and 
    \begin{equation*}
        X_{n_\varepsilon} \subset M(\gamma(s;y,\nu),\varepsilon) \subset  M(y_1, t + \varepsilon).
    \end{equation*}
    Then, recalling the definition of $X_n$ gives 
    \begin{equation*}
	M(y,s) \subset M(\Gamma, s - 1/n_{\varepsilon}) \cup M(y_1, t + \varepsilon).
    \end{equation*}

    Finally, assume that $(ii)$ holds. Since $d_g(\gamma(s;y,\nu), y) = s$ and $d_g(\gamma(s;y,\nu), \Gamma) = s$, it follows that $\gamma(s;y,\nu) \in M(y, s)$ and $\gamma(s;y,\nu) \notin M(\Gamma, s - \delta)$ for any $\delta>0$. Therefore, $(ii)$ implies that $\gamma(s;y,\nu) \in M(y_1,t + \varepsilon)$ for any $\varepsilon > 0$. Hence, we conclude that $d_g(\gamma(s;y,\nu), y_1) \leq t$.
\end{proof}
The preceding two lemmas demonstrated that the distances \eqref{distances} are recovered from the provided data. We also note that
\begin{equation*}
    M\setminus\mathcal{R} = \{y\in M: (s,y) \in \mathcal{N}_\mathcal{R} \text{ for some } s>0\},
\end{equation*}
see Lemma 2.10 in \cite{KatchalovKurylevLassas}. In other words, the boundary distance function 
\begin{equation*}
    d_g(y,z), \qquad y\in M\setminus\mathcal{R}, \textbf{ and } z\in \partial \mathcal{R},
\end{equation*}
is determined for the manifold $(M\setminus\mathcal{R}, g)$ with a smooth boundary.  Subsequently, the manifold $(M\setminus\mathcal{R}, g)$ can be constructed from the boundary distance function, as explained in Section 3.8 in \cite{KatchalovKurylevLassas}. This proves Theorem \ref{main_theorem}.

\bibliographystyle{plain}
\bibliography{references}

\setlength{\parskip}{0pt}





\end{document}